\documentclass[a4paper,12pt]{amsart}
\usepackage{amsmath}
\usepackage{amssymb}
\usepackage{mathrsfs}
\usepackage{enumerate}
\usepackage{ifthen}
\usepackage{graphicx}
\usepackage[T1]{fontenc} 
\usepackage{tabularx}
\usepackage{multirow}
\usepackage{color,soul}

\everymath{\displaystyle}

\nonstopmode \numberwithin{equation}{section}
\newtheorem{thm}{Theorem}[section]

\newtheorem{lem}{Lemma}[section]

\theoremstyle{definition}

\newcounter{minutes}
\divide\time by 60
\newcounter{hours}
\multiply\time by 60
\addtocounter{minutes}{-\time}

\newcounter {own}
\def\theown {\thesection       .\arabic{own}}

{\qed\bigskip}

\newcounter{alphabet}

\begin{document}
\title{ On close-to-convex functions}


\author{Md Nurezzaman}
\address{Md Nurezzaman, National Institute Of Technology Durgapur, West Bengal, India}
\email{nurezzaman94@gmail.com}


\def\thefootnote{}
\footnotetext{ {\tiny File:~\jobname.tex,
printed: \number\year-\number\month-\number\day,
          \thehours.\ifnum\theminutes<10{0}\fi\theminutes }
} \makeatletter\def\thefootnote{\@arabic\c@footnote}\makeatother

\begin{abstract}
We consider a new subclass $\widetilde{\mathcal{K}}_u$ of close-to-convex functions in the unit disk $\mathbb{D}:=\{z\in\mathbb{C}:|z|<1\}$. For this class, we obtain sharp estimates of the Fekete-Szeg\"{o} problem, growth and distortion theorem, radius of convexity and estimate of the pre-Schwarzian norm.
\end{abstract}

\thanks{}

\maketitle
\pagestyle{myheadings}
\markboth{ Md Nurezzaman }{On close-to-convex functions}

\section{Introduction}
Let $\mathcal{H}$ be the collection of all analytic functions in the open unit disk $\mathbb{D}:=\{z\in\mathbb{C}:|z|<1\}$. Also, let $\mathcal{A}$ be the subclass of $\mathcal{H}$ consisting of functions $f$ with the normalization $f(0)=f'(0)-1=0$, that is, having the Taylor series expansion \begin{align}\label{N-01} 
f(z)= z+\sum_{n=2}^{\infty}a_n z^n. 
\end{align}
Further, let $\mathcal{S}$ denote the collection of all univalent (or one-to-one) functions in $\mathcal{A}$. For $0\le\alpha<1$, a function $f\in\mathcal{A}$ is said to belong to the class $\mathcal{S}^*(\alpha)$ ( respectively, $\mathcal{C}(\alpha)$), known as starlike ( respectively, convex) functions of order $\alpha$, if $\mathrm{Re}~zf'(z)/f(z)>\alpha$ ( respectively, $\mathrm{Re}\left(1+zf''(z)/f'(z)\right)>\alpha$) in $\mathbb{D}$. Specifically, if $\alpha=0$, then $\mathcal{S}^*(0)=\mathcal{S}^*$ and $\mathcal{C}(0)=\mathcal{C}$ are the well-known classes of starlike and convex functions, respectively.\\
A function $f\in\mathcal{A}$ is said to be close-to-convex if the complement of the image-domain $f(\mathbb{D})$ in $\mathbb{C}$ is the union of rays that are disjoint (except that the origin of one ray may lie on another one of the rays). Analytically, a function $f\in \mathcal{A}$ is close-to-convex (see \cite{1983-Duren, 1952-Kaplan}) if and only if there exists a starlike function $g\in \mathcal{S}^*$ and a real number $\alpha\in(-\pi/2,\pi/2)$ such that
\begin{align*}
{\rm Re\,} \left(e^{i\alpha}\frac{zf'(z)}{g(z)}\right)>0,\quad z\in\mathbb{D}.
\end{align*}
Let $\mathcal{K}$ be the class of all close-to-convex functions, which was first introduced by Kaplan \cite{1952-Kaplan}.
A function $f\in\mathcal{A}$ is said to be strongly close-to-convex of order $\alpha>0$,
if there exist a real number $\beta\in(-\pi/2, \pi/2)$ and a starlike function $g\in\mathcal{S}^*$, such that
$$
\left|\arg\left(e^{i\beta}\frac{zf'(z)}{g(z)}\right)\right|<\frac{\alpha\pi}{2}\quad \mbox{ for }z\in\mathbb{D}.
$$
Let $\widetilde{\mathcal{K}}(\alpha)$ denote the class of all strongly close-to-convex functions ( see also \cite{1956-Reade}). We remark that $\widetilde{\mathcal{K}}(1)=\mathcal{K}$ and that $\widetilde{\mathcal{K}}(\alpha)$ is properly contained in $\mathcal{K}$ if $\alpha<1$. We can also take $\alpha>1$ but in this situation $f$ is not always univalent. For further information about these classes, we refer to \cite{1983-Duren, 1983-Goodman, 1984-Hallenbeck-Macgregor}.\\

In 1968, Singh \cite{1968-Singh} introduced and studied the class $\mathcal{S}_u^*$ consisting of function $f$ in $\mathcal{A}$ such that
\begin{align*}
\left|\frac{zf'(z)}{f(z)}-1\right|<1\quad \text{for}~z\in\mathbb{D}.
\end{align*}
It is easy to see that every function in $\mathcal{S}_u^*$ is also belongs to $\mathcal{S}^*$.  For $n=2,3,\dots,$ the functions $f_n(z)=ze^\frac{z^{n-1}}{n-1}$ belongs to the class $\mathcal{S}_u^*$. The functions $f_n$ plays important role for many extremal problems in $\mathcal{S}_u^*$. Singh \cite{1968-Singh} also obtained the distortion theorem, coefficient estimate and radius of convexity for the class $\mathcal{S}_u^*$. Recently, Allu et al. \cite{2020-Allu-Sokol-Thomas} introduced a close-to-convex analogue of $\mathcal{S}_u^*$ and the class is denoted as $\mathcal{K}_u$. A function $f$ of the form \eqref{N-01} in $\mathcal{A}$ belongs to $\mathcal{K}_u$ if there exists a starlike function $g\in\mathcal{S}^*$ such that
\begin{align*}
\left|\frac{zf'(z)}{g(z)}-1\right|<1\quad \text{for }~z\in\mathbb{D}.
\end{align*}
Clearly, every function in $\mathcal{K}_u$ is close-to-convex and hence univalent. For more details about the class $\mathcal{K}_u$ see \cite{2024-Ali-Nurezzaman, 2020-Allu-Sokol-Thomas}. Analogous to the class $\mathcal{K}_u$, we introduce a new class $\widetilde{\mathcal{K}}_u$. A function $f\in\mathcal{A}$ of the form \eqref{N-01} belongs to $\widetilde{\mathcal{K}}_u$, if there exists a function $g\in\mathcal{S}_u^*$ such that 
$$
\left|\frac{zf'(z)}{g(z)}-1\right|<1,\quad z\in\mathbb{D}.
$$
Since, the inclusion $\widetilde{\mathcal{K}}_u\subset \mathcal{K}_u\subset \mathcal{K}\subset\mathcal{S}$ is true, so any function $f$ in $\widetilde{\mathcal{K}}_u$ is univalent. Also, as the Koebe function $f(z)=z/(1-z)^2$ belongs to the class $\mathcal{S}^*$ but not in $\mathcal{S}_u^*$, so the inclusion $\widetilde{\mathcal{K}}_u\subset \mathcal{K}_u$ is proper.\\ 

For $\lambda\in\mathbb{C}$, the Fekete-Szeg\"{o} problem is to find the maximum value of the coefficient functional
\begin{align}\label{N-05}
\Phi_\lambda(f)=|a_3-\lambda a_2^2|,
\end{align}
when $f$ of the form \eqref{N-01} varies over a class of function $\mathcal{F}$. It is important to note that $\Phi_\lambda(f)$ behaves well with respect to rotation, namely, $\Phi_\lambda(R_\theta f)=\Phi_\lambda(f)$ for $\theta\in\mathbb{R}$. Here $R_\theta f$ denotes the rotation of $f$ by angle $\theta$, more precisely, $R_\theta f(z)=e^{-i\theta}f(e^{i\theta}z)$. Fekete and Szeg\"{o} \cite{1933-Fekete-Szego} proved one of the most important result in geometric function theory of complex analysis. In $1933$, Fekete and Szeg\"{o} \cite{1933-Fekete-Szego} proved that the functional $\Phi_\lambda(f)$ satisfies the following result by using L\"{o}ewner differential method
\begin{align*}
\max\limits_{f\in\mathcal{S}} \Phi_\lambda(f)=
\begin{cases}
3-4\lambda\quad &\mbox{~for~}\lambda\le0, \\
1+2e^{-\frac{2\lambda}{1-\lambda}}\quad &\mbox{~for~}0\leq\lambda\le1, \\
4\lambda-3\quad &\mbox{~for~}\lambda\ge1.
\end{cases}
\end{align*}
 Later, Pfluger \cite{1985-Pfluger} has considered the complex values of $\lambda$ and provided
$$
\max\limits_{f\in\mathcal{S}} \Phi_\lambda(f)=1+2\left|e^{-\frac{2\lambda}{1-\lambda}}\right|,\quad \lambda\in\mathbb{C}.
$$  
This inequality is sharp in the sense that for each $\lambda\in\mathbb{C}$ there exists a function in $\mathcal{S}$ such that equality holds.
In $1969$, Keogh et al. \cite{1969-Keogh-Merkes} obtained the sharp bound of $\Phi_\lambda(f)$ for functions in the classes $\mathcal{C},~\mathcal{S}^*$ and $\mathcal{K}$ with $\alpha=0$.
Later, in $1987$, Koepf \cite{1987-Koepf-1, 1987-Koepf-2} extended the above result over the classes $\mathcal{K}$ and $\widetilde{\mathcal{K}}(\alpha)$ and obtained the sharp bound of $\Phi_\lambda(f)$ for any $\lambda\in\mathbb{R}$. Ma-Minda \cite{1991-Ma-Minda,1992-Ma-Minda} studied the classes of strongly starlike functions, strongly convex functions and Ma-Minda starlike functions and obtained sharp estimates of $\Phi_\lambda(f),~\lambda\in\mathbb{R}$ for functions belongs to these classes. For $0\le\alpha\le1$, Gawad et al. \cite{1992-Abdel-Gawad} estimated that if a function $f$ belongs to the class $\widetilde{\mathcal{K}}(\alpha)$ with $\beta=0$, then the functional $\Phi_\lambda(f)$ has sharp bound for any $\lambda\in\mathbb{R}$. Later, London \cite{1993-London} extended the result over the class $\widetilde{\mathcal{K}}(\alpha)$ with $\beta=0$, and obtained sharp bound of $\Phi_\lambda(f)~(\lambda\in\mathbb{R})$ for any $\alpha>0$. For more information about the estimates of the functional $\Phi_\lambda(f)$, one can see the manuscripts of \cite{ 2007-Choi-Kim-Sugawa, 2010-Bhowmik, 2014-Kowalczyk-Lecko, 1997-Ma-Minda, 1974-Singh-Singh}.

\section{Some preliminary results}
Before we prove our main results, let us discuss some well known results that we will utilize throughout the article to derive our results as well as to find our extremal functions. Let, $\mathcal{B}$ be the subclass of $\mathcal{H}$ consisting of all functions $f$ in $\mathcal{H}$ with $|f(z)|<1$ for all $z\in\mathbb{D}$ and $\mathcal{B}_0$ be the subclass of $\mathcal{B}$ with $f(0)=0$. Functions in $\mathcal{B}_0$ are also known as Schwarz functions. If $\omega\in\mathcal{B}_0$, then $|\omega(z)|\le |z|$ and $|\omega'(0)|\le 1$. Further, in each of these inequalities, equality occurs if and only if $\omega(z)=e^{i\alpha}z$, $\alpha\in\mathbb{R}$. For functions in the class $\mathcal{B}$, the Schwarz-Pick lemma is a natural extension of Schwarz lemma. If $\omega\in\mathcal{B}$, then $|\omega'(z)|\le (1-|\omega(z)|^2)/(1-|z|^2)$ for $z\in\mathbb{D}$. Moreover, equality holds if and only if $\omega$ is a conformal self-map of $\mathbb{D}$. Here, a function $\omega\in\mathcal{B}$, is a conformal self-map of $\mathbb{D}$ if and only if $\omega$ is of the form
$$
\omega(z)=e^{i\alpha}\frac{z-a}{1-\bar{a}z},\quad 0\leq\alpha\leq2\pi \quad \mathrm{and} \quad a\in\mathbb{D}.
$$
One of the most important conformal self-map of $\mathbb{D}$ is the finite Blaschke product (see \cite[Page 265]{2001-Gamelin}). For $n\in\mathbb{N}$, Blaschke product of order $n$ is defined as the form
\begin{align*}
\omega(z)=e^{i\alpha}\left(\frac{z-a_1}{1-\bar{a_1}z}\right)\left(\frac{z-a_2}{1-\bar{a_2}z}\right)\cdots \left(\frac{z-a_n}{1-\bar{a_n}z}\right),
\end{align*}
where $\alpha\in[0,2\pi]$ and $a_1,a_2\ldots ,a_n\in\mathbb{D}$. In this article, the Blaschke product of order $2$ has been plays a crucial role for constructing extremal functions in the class $\widetilde{\mathcal{K}}_u$.

In the next lemma, we will discuss some preliminary results for the class $\mathcal{S}$. The following lemma due to Privalov \cite{1924-Privalov} will be helpful to prove our result.  
\begin{lem}\cite[Vol. \textbf{I}, Page 67]{1983-Goodman}\label{N-15}
Let $f\in\mathcal{S}$ and $z=re^{i\theta}\in\mathbb{D}$. If 
$$
m'(r)\le|f'(z)|\le M'(r),
$$
where $m'(r)$ and $M'(r)$ are real-valued functions of $r$ in $[0,1)$, then
$$
\int_0^rm'(t)dt\le|f(z)|\le\int_0^rM'(t)dt.
$$
\end{lem}
\vspace{4mm}

The following result was proved by Choi et al. \cite{2007-Choi-Kim-Sugawa}. But here we have written a part of it for our convenient.

\begin{lem}\label{N-17}\cite{2007-Choi-Kim-Sugawa}
For $A, B\in\mathbb{C}$ and $K, L, M\in\mathbb{R}$, let
\begin{align*}
\Omega_(A, B, K, L, M)=\max_{\substack{|u_1\leq 1\\ |v_1|\leq 1}}\left(|A|(1-|u_1|^2)+|B|(1-|v_1|^2)+|Ku_1^2+Lv_1^2+2Mu_1v_1|\right),
\end{align*}
Further consider the following three conditions involving $A, B, K, L, M$:
\begin{enumerate}
\item[\textbf{(A1)}] $|A|\geq \max\{|K|\sqrt{1-\dfrac{M^2}{KL}},|M|-|K|\}$;
\item[\textbf{(B1)}] $|B|\geq \max\{|L|\sqrt{1-\dfrac{M^2}{KL}},|M|-|L|\}$;
\item[\textbf{(B2)}] $|L|+|M|\leq |B| <|L|\sqrt{1-\dfrac{M^2}{KL}}$.
\end{enumerate}
If $KL\geq 0$, and $D=(|K|-|A|)(|L|-|B|)-M^2$ then
$$
\Omega(A, B, K, L, M)= \begin{cases}
|A|+|B| &\text{if }~|A|+|B|\ge |K|+|L|~\text{and}~D\ge0,
\\
|A|+|L|-\dfrac{M^2}{|K|-|L|} &\text{if }~|A|> |M|+|K|~\text{and}~D<0,\\
|B|+|K|-\dfrac{M^2}{|L|-|B|} &\text{if }~|B|> |M|+|L|~\text{and}~D<0,\\
|K|+2|M|+|L| &\text{otherwise}.
\end{cases}
$$
If $KL<0$, then $\Omega(A, B, K, L, M)=|A|+|B|+\max\{0,R\},$ where
\begin{align*}
R=\begin{cases}
0,\quad &\text{when A1 \& B1 holds }, \\
|L|-|B|+\dfrac{M^2}{|A|+|K|} ,\quad &\text{when A1 holds but B1 \& B2 does not hold}.
\end{cases}
\end{align*}
\end{lem}
\vspace{4mm}

\section{Growth and Distortion theorems}
Allu et al. \cite{2020-Allu-Sokol-Thomas} obtained sharp bound of growth and distortion for functions in the class $\mathcal{K}_u$. In the next theorem, we obtain sharp growth and distortion for functions in the class $\widetilde{\mathcal{K}}_u$.
\begin{thm}
If $f\in\widetilde{\mathcal{K}}_u$ and $z=re^{i\theta},~0\le r<1$, then 
\begin{align}\label{N-20}
re^{-r}\le |f(z)|\le re^r.
\end{align}
\begin{align}\label{N-25}
e^{-r}(1-r)\le |f'(z)|\le e^r(1+r),
\end{align}
\end{thm}
Moreover, all the inequalities are sharp.
\begin{proof}
Since $f\in\widetilde{\mathcal{K}}_u$, it follows that there exists a function $g\in\mathcal{S}_u^*$ and another function $\omega\in \mathcal{B}_0$ such that
\begin{align}\label{N-30}
f'(z)=\frac{g(z)}{z}\left(1+\omega(z)\right).
\end{align}
Since $g\in\mathcal{S}_u^*$, it follows from \cite[Theorem 2]{1968-Singh} with $z=re^{i\theta},0\le r<1$,
\begin{align}\label{N-35}
e^{-r}\le \left|\frac{g(z)}{z}\right|\le e^r,
\end{align} 
and also by Schwarz lemma, we have
\begin{equation}\label{N-40}
1-r\le|1+\omega(z)|\le1+r.
\end{equation}
Thus from \eqref{N-30}, using \eqref{N-35} and \eqref{N-40}, we immediately obtain \eqref{N-25}.

To show that both the inequalities in \eqref{N-25} are sharp, let us consider $f_2\in\widetilde{\mathcal{K}}_u$, given by
$$
f_2'(z)=\frac{g(z)}{z}\left(1+z\right),
$$
where $g(z)=ze^z\in\mathcal{S}_u^*$. Then, for $0\leq r<1$, on positive real axis
$$
|f_2'(r)|=e^r(1+r)
$$
and on negative real axis 
$$
|f_2'(-r)|=e^{-r}(1-r).
$$

Since $\widetilde{\mathcal{K}}_u\subset\mathcal{S}$, so the inequalities in \eqref{N-20} follows from Lemma \ref{N-15}. The estimates in \eqref{N-20} are sharp for the function $f_2(z)=ze^z$ in $\widetilde{\mathcal{K}}_u$.
\end{proof}
\vspace{4mm}

\section{Radius Of Convexity}
A number $s\in[0,1]$ is called the radius of convexity of a particular subclass $\mathcal{F}$ of $\mathcal{A}$ if $s$ is the largest number such that $$
\mathfrak{Re}\left(1+\frac{zf''(z)}{f'(z)}\right)>0,\quad \mathrm{for}\quad |z|<s
$$ for all $f$ in $\mathcal{F}$.

In $2020$, Allu et al. \cite{2020-Allu-Sokol-Thomas} prove that the radius of convexity of the class $\mathcal{K}_u$ is $1/3$. In the next theorem, we show that when $f\in\widetilde{\mathcal{K}}_u$, the radius of convexity of $\widetilde{\mathcal{K}}_u$ is $(3-\sqrt{5})/2$, which is larger than $1/3$.
\begin{thm}
The radius of convexity of $\widetilde{\mathcal{K}}_u$ is $(3-\sqrt{5})/2$.
\end{thm}
\begin{proof}
Since $f\in\widetilde{\mathcal{K}}_u$, it follows that there exists a function $g\in\mathcal{S}_u^*$ and another function $\omega\in \mathcal{B}_0$ such that
$$
f'(z)=\frac{g(z)}{z}\left(1+\omega(z)\right).
$$
Thus taking logarithmic derivative we have,
\begin{align}\label{N-45}
1+\frac{zf''(z)}{f'(z)}=\frac{zg'(z)}{g(z)}+\frac{z\omega'(z)}{1+\omega(z)}.
\end{align}
Now using Schwarz lemma, for $g\in\mathcal{S}_u^*$, we have
$$
\mathfrak{Re}\left\{\frac{zg'(z)}{g(z)}\right\}
\ge 1-|z|.
$$
Thus from \eqref{N-45}, for $z=re^{i\theta}$ and $0\le r<1$,
\begin{align*}
\mathfrak{Re}\left\{1+\frac{zf''(z)}{f'(z)}\right\}&\geq \mathfrak{Re}\left\{\frac{zg'(z)}{g(z)}\right\}-\left|\frac{z\omega'(z)}{1+\omega(z)}\right|\\&\geq1-|z|-\frac{|z|\left(1-|\omega(z)|^2\right)}{(1-|\omega(z)|)(1-|z|^2)}\\&\geq 1-r-\frac{r}{1-r}\\&=\frac{1-3r+r^2}{1-r}>0,
\end{align*}
when $0\le r<(3-\sqrt{5})/2$. Thus the radius of convexity for the class $\widetilde{\mathcal{K}}_u$ is at least $(3-\sqrt{5})/2$.

To show that this is the largest such radius, let us consider the function $f_2\in\widetilde{\mathcal{K}}_u$, given by
$$
f_2(z)=ze^z,\quad z\in\mathbb{D}.
$$
Therefore,
$$
1+\frac{zf_2''(z)}{f_2'(z)}=\frac{1+3z+z^2}{1+z}.
$$
For $z=-(3-\sqrt{5})/2$, we have
$$
\mathfrak{Re}\left\{1+\frac{zf_2''(z)}{f_2'(z)}\right\}=0.
$$ 
Thus the radius of convexity for the class $\widetilde{\mathcal{K}}_u$ is at least $(3-\sqrt{5})/2$.

\end{proof}

\section{Fekete-Szeg\"{o} Problem}
Recently, Ali et al. \cite{2024-Ali-Nurezzaman} obtained sharp estimates of $\Phi_\lambda(f),~\lambda\in\mathbb{R}$ completely for functions in the class $\mathcal{K}_{u}$. In the next theorem, we give sharp bound of $\Phi_\lambda(f),~\lambda\in\mathbb{R}$ for functions $f$ in $\widetilde{\mathcal{K}}_u$.
\begin{thm}\label{N-47}
Let $f\in \widetilde{\mathcal{K}}_u$ be given by \eqref{N-01}. Then for every $\lambda\in \mathbb{R}$
$$
\Phi_\lambda(f)=|a_3-\lambda a_2^2|\leq \begin{cases}
\dfrac{1}{2}-\lambda &\text{if}\quad \lambda\leq-\dfrac{1}{3},\\[3mm]
\dfrac{14-3\lambda}{6(4+3\lambda)} &\text{if}\quad -\dfrac{1}{3}\leq \lambda \leq \frac{1}{6},\\[3mm]
\dfrac{1}{2} &\text{if}\quad \frac{1}{6}\leq \lambda \leq 1,\\[3mm]
\lambda-\dfrac{1}{2} &\text{if}\quad \lambda\geq 1.
\end{cases}
$$
Moreover, all the inequalities are sharp.
\end{thm}

\begin{proof}
Let $f\in \widetilde{\mathcal{K}}_u$ be of the form \eqref{N-01}. Then there exists a function $g(z)=z+\sum\limits_{n=2}^{\infty}b_nz^n$ in $\mathcal{S}_u^* $ and another function $\omega_1(z)=\sum\limits
_{n=1}^{\infty}c_nz^n$ in $\mathcal{B}_0$ such that
\begin{align}\label{N-50}
f'(z)=\frac{g(z)}{z}\left(1+\omega_1(z)\right).
\end{align}
From \eqref{N-50}, comparing the coefficient of $z^2$ and $z^3$ on both sides,  we have
\begin{align}\label{N-55}
a_2=\frac{b_2}{2}+\frac{c_1}{2}\quad\text{and}\quad a_3=\frac{c_2}{3}+\frac{b_3}{3}+\frac{1}{3}b_2c_1.
\end{align}
Since $g\in \mathcal{S}_u^* $, it follows that there exists another $\omega_2(z)\in\mathcal{B}_0$ of the form $\omega_2(z)=\sum\limits_{n=1}^{\infty}d_nz^n$ such that
\begin{align}\label{N-60}
 \frac{zg'(z)}{g(z)}=1+ \omega_2(z).
\end{align}
On comparing the coefficients of $z^2$ and $z^3$ on both sides, we obtain
\begin{align}\label{N-65}
b_2=d_1\quad\text{and}\quad b_3=\frac{d_2+d_1^2}{2}.
\end{align}
From \eqref{N-55} and  \eqref{N-65}, one can easily obtain
\begin{align*}
a_2=\frac{c_1}{2}+\frac{d_1}{2}\quad\text{and}\quad a_3=\frac{c_2}{3}+\frac{d_2}{6}+\frac{d_1^2}{6}+\frac{1}{3}c_1d_1.
\end{align*}
Therefore, for any $\lambda\in \mathbb{R}$, we have
\begin{align*}
a_3-\lambda a_2^2=Ac_2+Bd_2+Kc_1^2+Ld_1^2+2Mc_1d_1,
\end{align*}
where
$$A=\frac{1}{3}, ~B=\frac{1}{6},~ K=-\frac{\lambda}{4},~ M=\frac{2-3\lambda}{12},~L=\frac{2-3\lambda}{12}.$$
Thus,
\begin{align*}
|a_3-\lambda a_2^2|&\leq |A||c_2|+|B||d_2|+|Kc_1^2+Ld_1^2+2Mc_1d_1|\\
&\leq |A|(1-|c_1|^2)+|B|(1-|d_1|^2)+|Kc_1^2+Ld_1^2+2Mc_1d_1|.\nonumber
\end{align*}
Now, we have to find the maximum value of $|a_3-\lambda a_2^2|$ when $|c_1|\leq 1,~ |d_1|\leq 1$. To do this we will use Lemma \ref{N-17} and we consider the following five cases.\\
\textbf{Case-1:} Let $\lambda\leq -\frac{1}{3}$. A simple calculation shows that
$$KL=-\frac{\lambda(2-3\lambda)}{48}\geq0,~ D=-\frac{1-6\lambda}{36}<0,~ |A|\le |M|+|K|,~ |B|\le |M|+|L|.$$
Therefore, from Lemma \ref{N-17}, we have
 $$|a_3-\lambda a_2^2|\leq|K|+2|M|+|L|=\dfrac{1}{2}-\lambda.$$
 
The inequality is sharp and the equality holds for the function $f_2\in \widetilde{\mathcal{K}}_u$ given by \eqref{N-50} and \eqref{N-60} with  
$$
\omega_1(z)=z\quad\mathrm{and}\quad\omega_2(z)=z,
$$ 
that is,
$$
f_2(z)=ze^z=z+z^2+\frac{1}{2}z^3+\cdots, \quad z\in\mathbb{D}.
$$
\textbf{Case-2:} Let $-\frac{1}{3}\leq\lambda\leq \frac{1}{6}$. For this, we consider the following subcases.\\
\textbf{Subcase-2a:} For $-\frac{1}{3}\leq\lambda\leq 0$, a simple calculation shows that
$$
KL=-\frac{\lambda(2-3\lambda)}{48}\geq0,~ D=-\frac{1-6\lambda}{36}<0.
$$
Also, 
$$
\frac{1}{3}=|A|>|M|+|K|=\frac{1-3\lambda}{6}\quad \mathrm{for}\quad -\frac{1}{3}\leq\lambda\leq 0.
$$
Thus from Lemma \ref{N-17}, we obtain
\begin{equation*}
|a_3-\lambda a_2^2|\leq|A|+|L|-\dfrac{M^2}{|K|-|A|}=\dfrac{14-3\lambda}{6(4+3\lambda)}.
\end{equation*}
\textbf{Subcase-2b:} For $0\leq\lambda\leq 1/6$, it is easy to show that $KL=\displaystyle -\frac{\lambda(2-3\lambda)}{48}<0$. So, from Lemma \ref{N-17}, we have
\begin{align}\label{N-70}
|a_3-\lambda a_2^2|\leq|A|+|B|+\max\{0,R\},
\end{align}
where $R$ can be obtained from Lemma \ref{N-17}.
For $0\leq\lambda\leq 1/6$, we have
\begin{align*}
|M|-|K|=\frac{1-3\lambda}{6}\le \frac{1}{3}=|A|
\end{align*}
and
\begin{align*}
&|K|\sqrt{1-\dfrac{M^2}{KL}}\leq|A|\\
&\iff\dfrac{\lambda}{4}\sqrt{\dfrac{2}{3\lambda}} \le\frac{1}{3}\\ 
&\iff \lambda\leq\frac{8}{3},
\end{align*}
which is true for $0\leq\lambda\leq 1/6$. Thus, the condition {\it (A1)} of Lemma \ref{N-17} is satisfied.\\

 Again, for $0\leq\lambda\leq 1/6$, we have
$$
|M|-|L|=0\leq \frac{1}{6}= |B|
$$
and
\begin{align*}
&|L|\sqrt{1-\dfrac{M^2}{KL}}\leq|B|\\&\iff\frac{2-3\lambda}{12}\sqrt{\dfrac{2}{3\lambda}}\leq\frac{1}{6}\\ &\iff 9\lambda^2-18\lambda+4\leq 0
\end{align*}
which is not true for any $0\leq \lambda\leq1/6$.  Thus, the condition {\it (B1)} of Lemma \ref{N-17} is not satisfied.

Further, for $0\leq\lambda\leq 1/6$,
$$|L|+|M|=\frac{2-3\lambda}{6}\ge \frac{1}{6}\ge |B|$$
and so, the condition {\it (B2)} of Lemma \ref{N-17} is not satisfied.

Therefore, by Lemma \ref{N-17} we have
$$R=|L|-|B|+\dfrac{M^2}{|A|+|K|}=\dfrac{2-3\lambda}{12}-\frac{1}{6}+\dfrac{(2-3\lambda)^2}{12(4+3\lambda)}=\frac{1-6\lambda}{12+9\lambda}\geq0,
$$
when $0\leq\lambda\leq 1/6$ and consequently, from \eqref{N-70}, we have
$$
  |a_3-\lambda a_2^2|\leq|A|+|B|+R=\frac{14-3\lambda}{6(4+3\lambda)}.
$$
Combining the \textbf{Subcase-2a} and \textbf{Subcase-2b}, we get
$$
  |a_3-\lambda a_2^2|\leq\frac{14-3\lambda}{6(4+3\lambda)}.
$$

The inequality is sharp and the equality holds for the function $f\in \widetilde{\mathcal{K}}_u$ given by \eqref{N-50} and \eqref{N-60} with
$$
\omega_1(z)=\frac{z(z+v_1)}{1+v_1z}\quad \text{and}\quad \omega_2(z)=z,
$$
where
$$ 
v_1=\frac{(2-3\lambda)}{4+3\lambda},
$$
 that is,
 $$
 f(z)=\int_0^z\frac{e^t(1+2v_1t+t^2)}{(1+v_1 t)}dt=z+\frac{3}{4+3\lambda}z^2+\frac{56+84\lambda-9\lambda^2}{6(4+3\lambda)^2}z^3+\cdots.
 $$
 For this function
 $$
 |a_3-\lambda a_2^2|=\frac{56+84\lambda-9\lambda^2}{6(4+3\lambda)^2}-\frac{9\lambda}{(4+3\lambda)^2}=\frac{14-3\lambda}{6(4+3\lambda)}.
 $$
\textbf{Case-3:} Let $1/6\leq\lambda\leq 1$. For this, we consider the following subcases.\\
\textbf{Subcase-3a:} For $1/6\leq\lambda\leq(3-\sqrt{5})/3$, it is easy to show that $KL=\displaystyle -\frac{\lambda(2-3\lambda)}{48}<0$. So, from Lemma \ref{N-17}, we have
\begin{align}\label{N-72}
|a_3-\lambda a_2^2|\leq|A|+|B|+\max\{0,R\},
\end{align}
where $R$ can be obtained from Lemma \ref{N-17}.
For $1/6\leq\lambda\leq(3-\sqrt{5})/3$, we have
\begin{align*}
|M|-|K|=\frac{1-3\lambda}{6}\le \frac{1}{3}=|A|
\end{align*}
and
\begin{align*}
&|K|\sqrt{1-\dfrac{M^2}{KL}}\leq|A|\\
&\iff\dfrac{\lambda}{4}\sqrt{\dfrac{2}{3\lambda}} \le\frac{1}{3}\\ 
&\iff \lambda\leq\frac{8}{3},
\end{align*}
which is true for $1/6\leq\lambda\leq(3-\sqrt{5})/3$. Thus, the condition {\it (A1)} of Lemma \ref{N-17} is satisfied.\\

 Again, for $1/6\leq\lambda\leq(3-\sqrt{5})/3$, we have
$$
|M|-|L|=0\leq \frac{1}{6}= |B|
$$
and
\begin{align*}
&|L|\sqrt{1-\dfrac{M^2}{KL}}\leq|B|\\&\iff\frac{2-3\lambda}{12}\sqrt{\dfrac{2}{3\lambda}}\leq\frac{1}{6}\\ &\iff 9\lambda^2-18\lambda+4\leq 0
\end{align*}
which is not true for any $1/6\leq\lambda\leq(3-\sqrt{5})/3$.  Thus, the condition {\it (B1)} of Lemma \ref{N-17} is not satisfied.

Further, for $1/6\leq\lambda\leq(3-\sqrt{5})/3$,
$$|L|+|M|=\frac{2-3\lambda}{6}\ge \frac{1}{6}\ge |B|$$
and so, the condition {\it (B2)} of Lemma \ref{N-17} is not satisfied.

Therefore, by Lemma \ref{N-17} we have
$$R=|L|-|B|+\dfrac{M^2}{|A|+|K|}=\dfrac{2-3\lambda}{12}-\frac{1}{6}+\dfrac{(2-3\lambda)^2}{12(4+3\lambda)}=\frac{1-6\lambda}{12+9\lambda}\leq0,$$
when $1/6\leq\lambda\leq(3-\sqrt{5})/3$ and consequently, from \eqref{N-72}, we have
$$
  |a_3-\lambda a_2^2|\leq|A|+|B|=\frac{1}{2}.
$$
\textbf{Subcase-3b:} For $(3-\sqrt{5})/3\leq\lambda\leq2/3$, it is easy to show that $KL=\displaystyle -\frac{\lambda(2-3\lambda)}{48}<0$. So, from Lemma \ref{N-17}, we have
\begin{align}\label{N-75}
|a_3-\lambda a_2^2|\leq|A|+|B|+\max\{0,R\},
\end{align}
where $R$ can be obtained from Lemma \ref{N-17}.
Now, for $(3-\sqrt{5})/3\leq\lambda\leq2/3$, we have
\begin{align*}
|M|-|K|=\frac{1-3\lambda}{6}\le \frac{1}{3}=|A|
\end{align*}
and
\begin{align*}
&|K|\sqrt{1-\dfrac{M^2}{KL}}\leq|A|\\
&\iff\dfrac{\lambda}{4}\sqrt{\frac{2}{3\lambda}} \le\frac{1}{3}\\ 
&\iff \lambda\leq\frac{8}{3},
\end{align*}
which is true for $(3-\sqrt{5})/3\leq\lambda\leq2/3$. Thus, the condition {\it (A1)} of Lemma \ref{N-17} is satisfied.\\

 Again, for $(3-\sqrt{5})/3\leq\lambda\leq2/3$, we have
$$
|M|-|L|=0\leq \frac{1}{6}= |B|
$$
and
\begin{align*}
&|L|\sqrt{1-\dfrac{M^2}{KL}}\leq|B|\\&\iff\frac{2-3\lambda}{12}\sqrt{\frac{2}{3\lambda}}\leq\frac{1}{6}\\ &\iff 9\lambda^2-18\lambda+4\leq 0
\end{align*}
which is true for any $(3-\sqrt{5})/3\leq\lambda\leq2/3$.  Thus, the condition {\it (B1)} of Lemma \ref{N-17} is satisfied.
Thus from Lemma \ref{N-17}, we get $R=0$. Therefore,  from \eqref{N-75}, we have
$$
|a_3-\lambda a_2^2|\leq|A|+|B|=\frac{1}{2}.
$$
\textbf{Subcase-3c:} For $2/3\leq \lambda\leq1$, a simple calculation shows that
$$
KL=\frac{\lambda(3\lambda-2)}{48}\geq0,~ D=\frac{1-\lambda}{12}\geq0, ~ |A|+|B|\ge |K|+|L|.
$$
Therefore, from Lemma \ref{N-17}, we have
$$
|a_3-\lambda a_2^2|\leq|A|+|B|=\frac{1}{2}.
$$
Combining the \textbf{Subcase-3a}, \textbf{Subcase-3b} and \textbf{Subcase-3c}, we get
$$
|a_3-\lambda a_2^2|\leq\frac{1}{2}.
$$

The inequality is sharp and the equality holds for the function $f_3\in \widetilde{\mathcal{K}}_u$ given by \eqref{N-50} and \eqref{N-60} with
$$
\omega_1(z)=z^2 \quad \text{and}\quad \omega_2(z)=z^2,
$$
that is,
$$
f_3(z)=\int_0^z e^{\frac{t^2}{2}}(1+t^2)dt=z+\frac{1}{2}z^3+\frac{1}{8}z^5+\cdots, \quad z\in\mathbb{D}.
$$

\textbf{Case-4:} Let $\lambda\geq1$. A simple calculation shows that
$$KL=\frac{\lambda(3\lambda-2)}{48}\geq0,~ D=\frac{1-\lambda}{12}<0, ~ |A|\le |M|+|K|, ~ |B|\le |M|+|L|.$$
Thus from Lemma \ref{N-17}, we obtain
$$|a_3-\lambda a_2^2|\leq|K|+2|M|+|L|=\lambda-\frac{1}{2}.$$

The inequality is sharp and the equality holds for the function $f_2\in \widetilde{\mathcal{K}}_u$ given by \eqref{N-50} and \eqref{N-60} with  
$$
\omega_1(z)=z\quad \mathrm{and}\quad\omega_2(z)=z,
$$
that is,
$$
f_2(z)=ze^z=z+z^2+\frac{1}{2}z^3+\cdots,\quad z\in\mathbb{D}.
$$
\end{proof}
\section{pre-Schwarzian norm}
Let $\mathcal{LU}$ represents the subclass of $\mathcal{H}$ that includes all locally univalent functions in $\mathbb{D}$, i.e., $\mathcal{LU}:=\{f\in\mathcal{H}:f'(z)\ne 0\text{ for all }z\in\mathbb{D}\}$. The pre-Schwarzian derivative of a locally univalent function $f\in\mathcal{LU}$ is defined as 
$$P_f(z)=\frac{f''(z)}{f'(z)},
$$
and the pre-Schwarzian norm, often known as the hyperbolic sup-norm, is defined as
\begin{align*}
||P_f||=\sup\limits_{z\in\mathbb{D}}(1-|z|^2)|P_f(z)|. 
\end{align*} 
In the theory of Teichm\"{u}ller spaces, this norm has important implications. It is well known that $||P_f||\leq 6$ for a univalent function $f$, and the bound is sharp. In contrast, if $f$ is univalent in $\mathbb{D}$, then $||P_f||\leq 1$ (see \cite{1972-Becker, 1984-Becker-Pommerenke}). Also, in 1976, Yamashita \cite{1976-Yamashita} proved that $||P_f ||$ is finite if and only if $f$ is uniformly locally univalent in $\mathbb{D}$. Moreover, if $||P_f||<2$, then $f$ is bounded in $\mathbb{D}$ (see \cite{2002-Kim-Sugawa}). In our next theorem, we establish sharp estimate of the pre-Schwarzian norm for functions in $\widetilde{\mathcal{K}}_{u}$.

\begin{thm}\label{N-85}
Let $f\in \widetilde{\mathcal{K}}_{u}$ be of the form \eqref{N-01}. Then
$$
 ||P_f||\le \dfrac{9}{4},
$$
and the estimate is sharp.
\end{thm}

\begin{proof}
Let $f\in \widetilde{\mathcal{K}}_{u}$ it follows that there exists a function $g\in\mathcal{S}_u^*$ and another function $\omega_1\in \mathcal{B}_0$ such that
 $$
 \frac{zf'(z)}{g(z)}=1+\omega_1(z).
 $$
Taking logarithmic derivative on both sides we have,
\begin{align}\label{N-90}
 \frac{f''(z)}{f'(z)}
 +\frac{1}{z}=\frac{g'(z)}{g(z)}+\frac{\omega_1'(z)}{1+\omega_1(z)}.
\end{align}
Since $g\in\mathcal{S}_{u}^*$, so there exists a function $\omega_2\in \mathcal{B}_{0}$ such that
 $$
 \frac{g'(z)}{g(z)}=\frac{1+\omega_2(z)}{z}.
 $$
Hence, from \eqref{N-90} we have
$$
\left|\frac{f''(z)}{f'(z)}\right|=\left|\frac{\omega_2(z)}{z}+\frac{\omega_1'(z)}{1+\omega_1(z)}\right|\le\frac{|\omega_2(z)|}{|z|}+\frac{|\omega_1'(z)|}{1-|\omega_1(z)|}.
$$
By Schwarz lemma and Schwarz-Pick lemma, we have

\begin{align*}
\left|\frac{f''(z)}{f'(z)}\right|\le1+\frac{1+|\omega_1(z)|}{1-|z|^2}\le\frac{2+|z|-|z|^2}{1-|z|^2}.
\end{align*}
Thus,
\begin{align*}
||P_f||&=\sup\limits_{z\in\mathbb{D}}(1-|z|^2)\left|\frac{f''(z)}{f'(z)}\right|\\
&\le\sup\limits_{0
\le|z|<1}(2+|z|-|z|^2)\\&=\frac{9}{4}.
\end{align*}

To show that the above inequality is sharp, let us consider $f_2\in\widetilde{\mathcal{K}}_{u}$ of the form
 $$
 f_2(z)=z e^z,\quad z\in\mathbb{D}.
 $$
Then,
$$
\frac{f_2''(z)}{f_2'(z)}=\frac{2+z}{1+z}.
$$
Therefore,
$$
||f_2||=\sup\limits_{z\in\mathbb{D}}(1-|z|^2)\left|\frac{2+z}{1+z}\right|
$$
On the negative side of real axis, we note that
$$
\sup\limits_{0\le r<1}(1+r)(2-r)=\frac{9}{4},
$$
 hence $||f_2||=9/4$.
\end{proof}

\vspace{4mm}
\noindent\textbf{Data availability:}
Data sharing not applicable to this article as no data sets were generated or analyzed during the current study.\vspace{4mm}
\\
\noindent\textbf{Acknowledgement:}
The author would like to acknowledge Dr. Md Firoz Ali for his guidance and assistance.

\end{document}